\documentclass{amsart}
\usepackage{amssymb}

\newcommand{\clconv}{\overline{\mathrm{conv}}}
\newcommand{\cov}{\mathrm{cov}}
\newcommand{\covJ}{\mathrm{cov}{}\kern-2pt{}_{\mathcal J}}
\newcommand{\e}{\varepsilon}
\newcommand{\IR}{\mathbb R}

\newcommand{\IZ}{\mathbb Z}
\newcommand{\w}{\omega}
\newcommand{\HH}{\mathcal H}
\newcommand{\F}{\mathcal F}
\newcommand{\C}{\mathcal C}
\newcommand{\Is}{\mathsf{Is}}
\newcommand{\is}{\mathsf{is}}
\newcommand{\I}{\mathcal I}
\newcommand{\U}{\mathcal U}
\newcommand{\supp}{\mathrm{supp}}
\newcommand{\PP}{\mathcal P}

\newtheorem{theorem}{Theorem}
\newtheorem{corollary}{Corollary}

\newtheorem{problem}{Problem}

\newtheorem{lemma}{Lemma}
\theoremstyle{definition}
\newtheorem{remark}{Remark}

\title[Difference sets in partitions of $G$-spaces and groups]{The covering number of the difference sets\\ in partitions of $G$-spaces and groups}
\author{Taras Banakh and Mikolaj Fr\c aczyk}
\address{T.Banakh: Ivan Franko National University of Lviv (Ukraine) and Jan Kochanowski University in Kielce (Poland)}
\email{t.o.banakh@gmail.com}
\address{M.Fr\c aczyk: Institute of Mathematics, Jagiellonian University, Krak\'ow (Poland) and Institut Galil\'ee, Universite Paris 13, Paris (France)}
\email{mikolaj.fraczyk@gmail.com}

\begin{document}
\begin{abstract} We prove that for every finite partition $G=A_1\cup\dots\cup A_n$ of a group $G$ either $\cov(A_iA_i^{-1})\le n$ for all cells $A_i$ or else $\cov(A_iA_i^{-1}A_i)<n$ for some cell $A_i$ of the partition. Here $\cov(A)=\min\{|F|:F\subset G,\;G=FA\}$ is the covering number of $A$ in $G$. A similar result is proved also of partitions of $G$-spaces. This gives two partial answers to a problem of Protasov posed in 1995.
\end{abstract}

\maketitle

This paper was motivated by the following problem posed by I.V.Protasov in Kourovka Notebook \cite{Kourov}.

\begin{problem}[Protasov, 1995]\label{prob1} Is it true that for any partition $G=A_1\cup\dots\cup A_n$ of a group $G$ some cell $A_i$ of the partition has $\cov(A_iA_i^{-1})\le n$?
\end{problem}

Here for a non-empty subset $A\subset G$ by
$$\cov(A)=\min\{|F|:F\subset G,\;\;G=FA\}$$we denote the {\em covering number} of $A$.

In fact, Protasov's Problem can be posed in a more general context of ideal $G$-spaces. Let us recall that a {\em $G$-space} is a set $X$ endowed with an action $G\times X\to X$, $(g,x)\mapsto gx$, of a group $G$. An {\em ideal $G$-space} is a pair $(X,\I)$ consisting of a $G$-space $X$ and a $G$-invariant Boolean ideal $\I\subset\mathcal B(X)$ in the Boolean algebra $\mathcal B(X)$ of all subsets of $X$. A  {\em Boolean ideal\/} on $X$ is a proper subfamily $\I\varsubsetneqq\mathcal B(X)$ such that for any  $A,B\in\mathcal I$ any subset $C\subset A\cup B$ belongs to $\I$. A Boolean ideal $\I$ is {\em $G$-invariant} if $\{gA:g\in G,\;A\in\I\}\subset\I$.
A Boolean ideal $\I\subset \mathcal B(G)$ on a group $G$ will be called {\em invariant} if $\{xAy:x,y\in G,\;\;A\in\I\}\subset \I$.
By $[X]^{<\w}$ and $[X]^{\le\w}$ we denote the families of all finite and countable subsets of a set $X$, respectively. The family $[X]^{<\w}$ (resp. $[X]^{\le \w}$) is a Boolean ideal on $X$ if $X$ is infinite (resp. uncountable).

For a subset $A\subset X$ of an ideal $G$-space $(X,\mathcal I)$ by
$$\Delta(A)=\{g\in G:gA\cap A\ne\emptyset\}\mbox{ \ and \ }\Delta_\I(A)=\{g\in G:gA\cap A\notin\I\}$$we denote the {\em difference set} and {\em $\I$-difference set} of $A$, respectively.

Given a Boolean ideal $\mathcal J$ on a group $G$ and two subsets $A,B\subset G$ we shall write $A=_{\mathcal J}B$ if the symmetric difference $A\triangle B=(A\setminus B)\cup(B\setminus A)$ belongs to the ideal $\mathcal J$. For a non-empty subset $A\subset G$ put $$\covJ(A)=\min\{|F|:F\subset G,\;FA=_{\mathcal J}G\}$$
be the {\em $\mathcal J$-covering number} of $A$. For the empty subset we put $\cov_{\mathcal J}(\emptyset)=\infty$ and assume that $\infty$ is larger than any cardinal number.

Observe that for the left action of the group $G$ on itself we get $\Delta(A)=AA^{-1}$  for every subset $A\subset G$. That is why Problem~\ref{prob1} is a partial case of the following general problem.

\begin{problem}\label{prob2} Is it true that for any partition $X=A_1\cup\dots\cup A_n$ of an ideal  $G$-space $X$ some cell $A_i$ of the partition has $\cov(\Delta_\I(A_i))\le n$?
\end{problem}

This problem has an affirmative answer for $G$-spaces with amenable acting group $G$, see \cite[4.3]{BPS}. The paper \cite{BPS} gives a survey of available partial solutions of Protasov's Problems~\ref{prob1} and \ref{prob2}. Here we mention the following result of Banakh, Ravsky and Slobodianiuk \cite{BRS}.

\begin{theorem} For any partition $X=A_1\cup\dots\cup A_n$ of an ideal $G$-space $(X,\I)$ some cell $A_i$ of the partition has $$\cov(\Delta_\I(A_i))\le \max_{0<k\le n}\sum_{p=0}^{n-k}k^p\le n!$$
\end{theorem}

In this paper we shall give another two partial solutions to Protasov's Problems~\ref{prob1} and \ref{prob2}.

\begin{theorem}\label{t1} For any partition $X=A_1\cup\dots\cup A_n$ of an ideal $G$-space $(X,\I)$ either
\begin{itemize}
\item $\cov(\Delta_\I(A_i))\le n$ for all cells $A_i$ or else
\item $\covJ(\Delta_\I(A_i))<n$ for some cell $A_i$ and some $G$-invariant ideal $\mathcal J\not\ni \Delta_\I(A_i)$ on $G$.
\end{itemize}
\end{theorem}

\begin{corollary}\label{c1} For any partition $X=A_1\cup\dots\cup A_n$ of an ideal $G$-space $(X,\I)$ either $\cov(\Delta_\I(A_i))\le n$ for all cells $A_i$ or else  $\cov(\Delta_\I(A_i)\cdot\Delta_\I(A_i))<n$ for some cell $A_i$.
\end{corollary}

\begin{proof} By Theorem~\ref{t1}, either $\cov(\Delta_\I(A_i))\le n$ for all cells $A_i$ or else there is a cell $A_i$ of the partition such that $\covJ(\Delta_\I(A_i))< n$ for some $G$-invariant ideal $\mathcal J\varsubsetneqq\PP(G)$. In the first case we are done. In the second case we can find a subset $F\subset G$ of cardinality $|F|<n$ such that $F\cdot\Delta_\I(A_i)=_{\mathcal J}G$. It follows from $G\notin\mathcal J\ni G\setminus (F\cdot \Delta_\I(A_i))$ that $F\cdot\Delta_\I(A_i)\notin\mathcal J$ and hence $\Delta_\I(A_i)\notin \mathcal J$ by the $G$-invariance of $\mathcal J$. Then for every $x\in G$ the shift $x\Delta_\I(A_i)$ does not belong to the ideal $\mathcal J$ and hence intersects the set $F\cdot\Delta_\I(A_i)$. So $x\in F\cdot\Delta_\I(A_i)\cdot\Delta_\I(A_i)^{-1}=F\cdot\Delta_\I(A_i)\cdot\Delta_\I(A_i)$ and $\cov(\Delta_\I(A_i)\cdot\Delta_\I(A_i))\le|F|\le n$.
\end{proof}

For groups $G$ (considered as $G$-spaces endowed with the left action of $G$ on itself), we can prove a bit more:

\begin{theorem}\label{t2} Let $G$ be a group and $\I$ be an invariant Boolean ideal on $G$ with $[G]^{\le\w}\not\subset \I$. For any partition $G=A_1\cup\dots\cup A_n$ of $G$ either
\begin{itemize}
\item $\cov(\Delta_\I(A_i))\le n$ for all cells $A_i$ or else
\item $\covJ(\Delta_\I(A_i))<n$ for some cell $A_i$ and for some $G$-invariant Boolean ideal $\mathcal J\not\ni A_i^{-1}$ on $G$.
\end{itemize}
\end{theorem}

\begin{corollary}\label{c2} For any partition $G=A_1\cup\dots\cup A_n$ of a group $G$ either $\cov(A_iA_i)\le n$ for all cells $A_i$ or else $\cov(A_iA_i^{-1}\kern-1pt A_i)<n$ for some cell $A_i$ of the partition.
\end{corollary}

\begin{proof} On the group $G$ consider the trivial ideal $\I=\{\emptyset\}$. By Theorem~\ref{t2}, either $\cov(A_iA_i^{-1})\le n$ for all cells $A_i$ or else $\cov_{\mathcal J}(A_iA_i^{-1})<n$ for some cell $A_i$ and some $G$-invariant ideal $\mathcal J\not\ni A_i^{-1}$ on $G$. In the first case we are done. In the second case, choose a finite subset $F\subset G$ of cardinality $|F|<n$ such that the set $FA_iA_i^{-1}=_{\mathcal J}G$.
Since $A_i^{-1}\notin\mathcal J$, for every $x\in G$ the set $xA_i^{-1}$ intersects $FA_iA_i^{-1}$ and thus $x\in FA_iA_i^{-1}\kern-1pt A_i$ and $\cov(A_iA_i^{-1}\kern-1pt A_i)\le|F|<n$.
\end{proof}

Taking into account that the ideal $\mathcal J$ appearing in Theorem~\ref{t2} is $G$-invariant but not necessarily invariant, we can ask the following question.

\begin{problem} Is it true that for any partition $G=A_1\cup\dots\cup A_n$ of a group $G$ some cell $A_i$ of the partition has $\covJ(A_iA_i^{-1})\le n$ for some invariant Boolean ideal $\mathcal J$ on $G$?
\end{problem}

\section{Minimal measures on $G$-spaces}

Theorems~\ref{t1} and \ref{t2} will be proved with help of minimal probability measures on $X$ and right quasi-invariant idempotent measures on $G$.

For a $G$-space $X$ by $P(X)$ we denote the (compact Hausdorff) space of all finitely additive probability measures on $X$. The action of the group $G$ on $X$ extends to an action of the convolution semigroup $P(G)$ on $P(X)$: for two measures $\mu\in P(G)$ and $\nu\in P(X)$ their convolution is defined as the measure $\mu*\nu\in P(X)$ assigning to each bounded function $\varphi:X\to\IR$ the real number
$$\mu*\nu(\varphi)=\int_G\int_X\varphi(g^{-1}x)\,d\nu(x)\, d\mu(g).$$ The convolution map $*:P(G)\times P(X)\to P(X)$ is right-continuous in the sense that for any fixed measure $\nu\in P(X)$ the right shift  $P(G)\to P(X)$, $\mu\mapsto\mu*\nu$, is continuous. This implies that the $P(G)$-orbit $P(G)*\nu=\{\mu*\nu:\mu\in P(G)\}$  of $\nu$ coincides with the closure $\clconv(G\cdot \nu)$ of the convex hull of the $G$-orbit $G\cdot \nu$ of  $\nu$ in $P(X)$.

A measure $\mu\in P(X)$ will be called {\em minimal\/} if for any measure $\nu\in P(G)*\mu$ we get $P(G)*\nu=P(G)*\mu$.  The Zorn's Lemma combined with the compactness of the orbits implies that the orbit $P(G)*\mu$ of each measure $\mu\in P(X)$ contains a minimal measure.

It follows from Day's Fixed Point Theorem \cite[1.14]{Pat} that for a $G$-space $X$ with amenable acting group $G$ each minimal measure $\mu$ on $X$ is $G$-invariant, which implies that the set $\clconv(G\cdot\mu)$ coincides with the singleton $\{\mu\}$.

For an ideal $G$-space $(X,\I)$ let $P_\I(X)=\{\mu\in P(X):\forall A\in \I\;\;\mu(A)=0\}$.

\begin{lemma}\label{l1} For any ideal $G$-space $(X,\I)$ the set $P_\I(X)$ contains some minimal probability measure.
\end{lemma}

\begin{proof} Let $\U$ be any ultrafilter on $X$, which contains the filter $\F=\{F\subset X:X\setminus F\in\I\}$. This ultrafilter $\U$ can be identified with the 2-valued measure $\mu_\U:\mathcal B(X)\to\{0,1\}$ such that $\mu_\U^{-1}(1)=\U$. It follows that $\mu_\U(A)=0$ for any subset $A\in\I$.
In the $P(G)$-orbit $P(G)*\mu_\U$ choose any minimal measure $\mu=\nu*\mu_\U$ and observe that for every $A\in\I$ the $G$-invariance of the ideal $\I$ implies $\mu(A)=\int_G\mu_\U(x^{-1}A)\,d\nu(x)=0$.
So, $\mu\in P_\I(X)$.
\end{proof}

For a subset $A$ of a group $G$ put
$$\Is_{12}(A)=\inf_{\mu\in P(G)}\sup_{y\in G}\mu(Ay).$$

\begin{lemma}\label{l2} If a subset $A$ of a group $G$ has $\Is_{12}(A)=1$, then $\cov(G\setminus A)\ge\w$.
\end{lemma}

\begin{proof} It suffices to show that $G\ne F(G\setminus A)$ for any finite set $F\subset G$. Consider the uniformly distributed measure $\mu=\frac1{|F|}\sum_{x\in F}\delta_{x^{-1}}$ on the set $F^{-1}$. Since $\Is_{12}(A)=1$, for the measure $\mu$ there is a point $y\in G$ such that $1-\frac1{|F|}<\mu(Ay)=\frac1{|F|}\sum_{x\in F}\delta_{x^{-1}}(Ay)$, which implies that $\mu(Ay)=1$ and $\supp(\mu)=F^{-1}\subset Ay$. Then $F^{-1}y^{-1}\cap (G\setminus A)=\emptyset$ and $y^{-1}\notin F(G\setminus A)$.
\end{proof}

\begin{remark} By Theorem 3.8 of \cite{Ban}, for every subset $A$ of a group $G$ we get $\Is_{12}(A)=1-\is_{21}(G\setminus A)$ where $\is_{21}(B)=\inf_{\mu\in P_\w(G)}\sup_{x\in G}\mu(xB)$ for $B\subset G$ and $P_\w(G)$ denotes the set of finitely supported probability measures on $G$.
\end{remark}

For a probability measure $\mu\in P(X)$ on a $G$-space $X$ and a subset $A\subset X$ put
$$\bar\mu(A)=\sup_{x\in G}\mu(xA).$$

\section{A density version of Theorem~\ref{t1}}

In this section we shall prove the following density theorem, which will be used in the proof of Theorem~\ref{t1} presented in the next section.

\begin{theorem}\label{t3} Let $(X,\I)$ be an ideal $G$-space and $\mu\in P_\I(X)$ be a minimal measure on $X$. If some subset $A\subset X$ has $\bar\mu(A)>0$, then the $\I$-difference set $\Delta_\I(A)$ has $\mathcal J$-covering number $\covJ(\Delta_\I(A))\le1/\bar\mu(A)$ for some $G$-invariant ideal $\mathcal J\not\ni\Delta_\I(A)$ on $G$.
\end{theorem}

\begin{proof} By the compactness of $P(G)*\mu=\clconv(G\cdot \mu)$, there is a measure $\mu'\in P(G)*\mu\subset P_\I(X)$ such that $\mu'(A)=\sup\{\nu\in P(G)*\mu:\nu(A)\}=\bar\mu(A)$. We can replace the measure $\mu$ by $\mu'$ and assume that $\mu(A)=\bar\mu(A)$.
Choose a positive $\e$ such that $\big\lfloor \frac1{\bar\mu(A)-\e}\big\rfloor=\big\lfloor\frac1{\bar\mu(A)}\big\rfloor$, where $\lfloor r\rfloor=\max\{n\in\IZ:n\le r\}$ denotes the integer part of a real number $r$.

 Consider the set $L=\{x\in G:\mu(xA)>\bar\mu(A)-\e\}$ and choose a maximal subset $F\subset L$ such that $\mu(xA\cap yA)=0$ for any distinct points $x,y\in L$. The additivity of the measure $\mu$ implies that $1\ge\sum_{x\in F}\mu(xA)>|F|(\bar\mu(A)-\e)$ and hence $|F|\le\lfloor\frac1{\bar\mu(A)-\e}\rfloor=\lfloor\frac1{\bar\mu(A)}\rfloor\le \frac1{\bar\mu(A)}$.
By the maximality of $F$, for every $x\in L$ there is $y\in L$ such that $\mu(xA\cap yA)>0$. Then $xA\cap yA\notin\I$ and $y^{-1}x\in\Delta_\I(A)$. It follows that $x\in y\cdot\Delta_\I(A)\subset F\cdot\Delta_\I(A)$ and $L\subset F\cdot \Delta_\I(A)$.

We claim that $\Is_{12}(L)=1$. Given any measure $\nu\in P(G)$, consider the measure $\nu^{-1}\in P(G)$ defined by $\nu^{-1}(B)=\nu(B^{-1})$ for every subset $B\subset G$. By the minimality of $\mu$, we can find a measure $\eta\in P(G)$ such that $\eta*\nu^{-1}*\mu=\mu$. Then
$$
\begin{aligned}
&\bar\mu(A)=\mu(A)=\eta*\nu^{-1}*\mu(A)=\int_G\mu(x^{-1}A)d\eta*\nu^{-1}(x))\le\\
 &\le(\bar\mu(A)-\e)\cdot\eta*\nu^{-1}\big(\{x\in G:\mu(x^{-1}A)\le\bar\mu(A)-\e\}\big)+\bar\mu(A)\cdot\eta*\nu^{-1}\big(\{x\in G:\mu(x^{-1}A)>\bar\mu(A)-\e\}\big)\le\\
 &\le (\bar\mu(A)-\e)\cdot\big(1-\eta*\nu^{-1}(L^{-1})\big)+\bar\mu(A)\cdot\eta*\nu^{-1}(L^{-1})\le\bar\mu(A)
\end{aligned}
$$
implies that $\eta*\nu^{-1}(L^{-1})=1$.
It follows from $$1=\eta*\nu^{-1}(L^{-1})=\int_G\nu^{-1}(y^{-1}L^{-1})d\eta(y)$$ that for every $\delta>0$ there is a point $y\in G$ such that $\nu(Ly)=\nu^{-1}(y^{-1}L^{-1})>1-\delta$.
So, $\Is_{12}(L)=1$.

By Lemma~\ref{l2}, the family $\mathcal J=\{B\subset G:\exists E\in[G]^{<\w}\;B\subset E(G\setminus L)\}$ is a $G$-invariant ideal on $G$, which does not contain the set $L\subset F\cdot\Delta_\I(A_i)$ and hence does not contain the set $\Delta_\I(A_i)$. It follows that $\covJ(\Delta_\I(A_i))\le |F|\le1/\bar\mu(A)$.
\end{proof}

\section{Proof of Theorem~\ref{t1}}

Let $X=A_1\cup\dots\cup A_n$ be a partition of an ideal $G$-space $(X,\I)$.
By Lemma~\ref{l1}, there exists a minimal probability measure $\mu\in P(X)$ such that $\I\subset\{A\in\mathcal B(G):\mu(A)=0\}$.

For every $i\in\{1,\dots,n\}$ consider the number $\bar\mu(A_i)=\sup_{x\in G}\mu(xA)$ and observe that $\sum_{i=1}^n\bar\mu(A_i)\ge 1$. There are two cases.
\smallskip

1) For every $i\in\{1,\dots,n\}$ \ $\bar\mu(A_i)\le \frac1n$. In this case for every $x\in G$ we get $$1=\sum_{i=1}^n\mu(xA_i)\le\sum_{i=1}^n\bar\mu(A_i)\le n\cdot\frac1n=1$$ and hence $\mu(xA_i)=\frac1n$ for every $i\in\{1,\dots,n\}$. For every $i\in\{1,\dots,n\}$ fix a maximal subset $F_i\subset G$ such that $\mu(xA_i\cap yA_i)=0$ for any distinct points $x,y\in F_i$.  The additivity of the measure $\mu$ implies that $1\ge \sum_{x\in F_i}\mu(xA_i)\ge|F_i|\frac1n$ and hence $|F_i|\le n$. By the maximality of $F_i$, for every $x\in G$ there is a point $y\in F_i$ such that $\mu(xA_i\cap yA_i)>0$ and hence $xA_i\cap yA_i\notin\I$. The $G$-invariance of the ideal $\I$ implies that $y^{-1}x\in \Delta_\I(A_i)$ and so $x\in y\cdot\Delta_\I(A_i)\subset F_i\cdot\Delta_\I(A_i)$. Finally, we get $G=F_i\cdot\Delta_I(A_i)$ and $\cov(\Delta_\I(A_i))\le|F_i|\le n$.
\smallskip

2) For some $i$ we get $\bar\mu(A_i)>\frac1n$. In this case Theorem~\ref{t3} guarantees that $\covJ(\Delta_\I(A_i))\le1/\bar\mu(A_i)<n$ for some $G$-invariant ideal $\mathcal J\not\ni\Delta_\I(A_i)$ on $G$.

\section{Applying idempotent quasi-invariant measures}

In this section we develop a technique involving idempotent right quasi-invariant measures, which will be used in the proof of Theorem~\ref{t2} presented in the next section.

A measure $\mu\in P(G)$ on a group $G$ will be called {\em right quasi-invariant} if for any $y\in G$ there is a positive constant $c>0$ such that $c\cdot\mu(Ay)\le\mu(A)$ for any subset $A\subset G$.

For an ideal $G$-space $(X,\I)$ and a measure $\mu\in P(X)$ consider the set
$$P_\I(G;\mu)=\{\lambda\in P(G):\forall g\in G\;\;\lambda*\delta_g*\mu\in P_\I(X)\}$$and observe that it is closed and convex in the compact Hausdorff space $P(G)$.

\begin{lemma}\label{l3} Let $(X,\I)$ be an ideal $G$-space with countable acting group $G$. If for some measure $\mu\in P(X)$ the set $P_\I(G;\mu)$ is not empty, then it contains a right quasi-invariant idempotent measure $\nu\in P_\I(G;\mu)$.
\end{lemma}

\begin{proof} Choose any strictly positive function $c:G\to (0,1]$ such that $\sum_{g\in G}c(g)=1$ and
consider the $\sigma$-additive probability measure $\lambda=\sum_{g\in G}c(g)\delta_{g^{-1}}\in P(G)$.
On the compact Hausdorff space $P(G)$ consider the right shift $\Phi:P(G)\to P(G)$, $\Phi:\nu\mapsto\nu*\lambda$.

We claim that $\Phi(P_\I(G;\mu))\subset P_\I(G;\mu)$. Given any measure $\nu\in P_\I(G;\mu)$ we need to check that $\Phi(\nu)=\nu*\lambda\in P_\I(G;\mu)$, which means that $\nu*\lambda*\delta_x*\mu\in P_\I(X)$ for all $x\in G$. It follows from $\nu\in P_\I(G;\mu)$ that $\nu*\delta_{g^{-1}x}*\mu\in P_\I(X)$. Since the set $P_\I(X)$ is closed and convex in $P(X)$, we get
$$\nu*\lambda*\delta_x*\mu=\sum_{g\in G}c(g)\cdot\nu*\delta_{g^{-1}}*\delta_x*\mu=\sum_{g\in G}\nu*\delta_{g^{-1}x}*\mu\in P_\I(X).$$

So, $\Phi(P_\I(G;\mu))\subset P_\I(G;\mu)$ and by Schauder Fixed Point Theorem, the continuous map $\Phi$ on the non-empty compact convex set $P_\I(G;\mu)\subset P(G)$ has a fixed point, which implies that the closed set $S=\{\nu\in P_\I(G;\mu):\nu*\lambda=\nu\}$ is not empty. It is easy to check that $S$ is a subsemigroup of the convolution semigroup $(P(G),*)$. Being a compact right-topological semigroup, $S$ contains an idempotent $\nu\in S\subset P_\I(G;\mu)$ according to Ellis Theorem \cite[2.6]{HS}. Since $\nu*\lambda=\nu$, for every $A\subset G$ and $x\in G$ we get
$$\nu(A)=\nu*\lambda(A)=\sum_{g\in G}c(g)\cdot\nu*\delta_{g^{-1}}(A)=\sum_{g\in G}c(g)\cdot\nu(Ag)\ge c(x)\cdot\nu(Ax),$$
which means that $\nu$ is right quasi-invariant.
\end{proof}

\begin{remark} Lemma~\ref{l3} does not hold for uncountable groups, in particular for the free group $F_{\alpha}$ with uncountable set $\alpha$ of generators. This group admits no right quasi-invariant measure. Assuming conversely that some measure $\mu\in P(F_\alpha)$ is right quasi-invariant, fix a generator $a\in \alpha$ and consider the set $A$ of all reduced words $w\in F_\alpha$ that end with $a^n$ for some $n\in\IZ\setminus\{0\}$. Observe that $F_\alpha=Aa\cup A$ and hence $\mu(A)>0$ or $\mu(Aa)>0$. Since $\mu$ is right quasi-invariant both cases imply that $\mu(A)>0$ and then $\mu(Ab)>0$ for any generator $b\in\alpha\setminus\{a\}$. But this is impossible since the family $(Ab)_{b\in\alpha\setminus\{a\}}$ is disjoint and uncountable.
\end{remark}

In the following lemma for a measure $\mu\in P(X)$ we put $\bar\mu(A)=\sup_{x\in G}\mu(xA)$.

\begin{lemma}\label{l4} Let $(X,\I)$ be an ideal $G$-space and $\mu\in P(X)$ be a measure on $X$ such that the set $P_\I(G;\mu)$ contains an idempotent right quasi-invariant measure $\lambda$. For a subset $A\subset X$ and numbers $\delta\le \e<\sup_{x\in G}\lambda*\mu(xA)$ consider the sets $M_\delta=\{x\in G:\mu(xA)>\delta\}$ and $L_\e=\{x\in G:\lambda*\mu(xA)>\e\}$. Then:
\begin{enumerate}
\item $\lambda(gM_\delta^{-1})>(\e-\delta)/(\bar\mu(A)-\delta)$ for any point $g\in L_\e$;
\item the set $M_\delta$ does not belong to the $G$-invariant Boolean ideal $\mathcal J_\delta\subset \PP(G)$ generated by $G\setminus L_\delta$;
\item $\cov_{\mathcal J_\delta}(\Delta_\I(A))<1/\delta$.
\end{enumerate}
\end{lemma}

\begin{proof} Consider the measure $\nu=\lambda*\mu$ and put $\bar\nu(A)=\sup_{x\in G}\nu(xA)$ for a subset $A\subset X$.
\smallskip

1. Fix a point $g\in L_\e$ and observe that
$$
\begin{aligned}
\e&<\lambda*\mu(gA)=\int_G\mu(x^{-1}gA)d\lambda(x)\le\\
&\le\delta\cdot\lambda(\{x\in G:\mu(x^{-1}gA)\le\delta\})+\bar\mu(A)\cdot\lambda(\{x\in G:\mu(x^{-1}gA)>\delta\})=\\
&=\delta\cdot(1-\lambda(gM_\delta^{-1}))+\bar\mu(A)\lambda(gM_\delta^{-1})=\delta+(\bar\mu(A)-\delta)\lambda(gM_\delta^{-1})
\end{aligned}
$$
which implies $\lambda(gM_\delta^{-1})>\gamma:=\frac{\e-\delta}{\bar\mu(A)-\delta}$.
\smallskip

2. To derive a contradiction, assume that the set $M_\delta$ belongs to the $G$-invariant ideal generated by $G\setminus L_\delta$ and hence $M_\delta\subset E(G\setminus L_\delta)$ for some finite subset $E\subset G$. Then
$$M_\delta\subset E(G\setminus L_\delta)=G\setminus \bigcap_{e\in E}eL_\delta.$$

Choose an increasing number sequence $(\e_k)_{k=0}^\infty$ such that $\delta<\e\le\e_0$ and $\lim_{k\to\infty}\e_k=\bar\nu(A)$.
For every $k\in\w$ fix a point $g_k\in L_{\e_k}$. The preceding item applied to the measure $\nu$ and set $L_\delta$ (instead of $\mu$ and $M_\delta$) yields the lower bound
$$\lambda(g_kL_\delta^{-1})>\frac{\e_k-\delta}{\bar\nu(A)-\delta}$$
for every $k\in \w$. Then $\lim_{k\to\infty}\lambda(g_kL_\delta^{-1})=1$ and
hence $\lim_{k\to\infty}\lambda(z_kL_\delta^{-1}g)=1$ for every $g\in G$ by the right quasi-invariance of the measure $\lambda$. Choose $k$ so large that $\lambda(z_kL_\delta^{-1}g^{-1})>1-\frac1{|E|}\gamma$ for all $g\in E$.
Then the set $\bigcap_{g\in E}z_kL_\delta^{-1}g^{-1}$ has measure $>1-\gamma$ and hence it intersects the set $z_kM_a^{-1}$ which has measure $\lambda(z_kM_a)\ge \gamma$. Consequently, the set $M_a^{-1}$ intersects $\bigcap_{g\in E}L_\delta^{-1}g^{-1}$, and the set $M_a$ intersects $\bigcap_{g\in E}gL=G\setminus\big( E(G\setminus L_\delta)\big)$, which contradicts the choice of the set $E$.
\smallskip

3. To show that $\cov_{\mathcal J_\delta}(\Delta_\I(A))\le1/\delta$, fix a maximal subset $F\subset L_\delta$ such that $\nu(xA\cap yA)=0$ for any distinct points $x,y\in L_\delta$. The additivity of the measure $\nu$ guarantees that $1\ge\sum_{x\in F}\nu(xA)>|F|\cdot\delta$ and hence $|F|<1/\delta$. On the other hand, the maximality of $F$ guarantees that for every $x\in F$ there is $y\in L_\delta$ such that $\nu(xA\cap yA)>0$ and hence $xA\cap yA\notin\I$ and $y^{-1}x\in\Delta_\I(A)$. Then $x\in y\cdot\Delta_\I(A)\subset F\cdot\Delta_\I(A)$ and hence $L_\delta\subset F\cdot \Delta_\I(A)$. The inclusion $G\setminus \big(F\cdot\Delta_\I(A)\big)\subset G\setminus L_\delta\in\mathcal J_\delta$ implies $\cov_{\mathcal J_\delta}(F\cdot \Delta_\I(A))\le |F|<1/\delta$.
\end{proof}

\begin{corollary}\label{c3} Let $(X,\I)$ be an ideal $G$-space with countable acting group $G$ and $\mu\in P(X)$ be a measure on $X$ such that the set $P_\I(G;\mu)$ is not empty. For any partition $X=A_1\cup\dots\cup A_n$ of $X$ either:
\begin{enumerate}
\item $\cov(\Delta_\I(A_i))\le n$ for all cells $A_i$ or else
\item $\covJ(\Delta_\I(A_i))<n$ for some cell $A_i$ and some $G$-invariant Boolean ideal $\mathcal J\subset\PP(G)$ such that $\{x\in G:\mu(xA)>\frac1n\}\notin\mathcal J$.
\end{enumerate}
\end{corollary}

\begin{proof} By Lemma~\ref{l3}, the set $P_\I(G;\mu)$ contains an idempotent right quasi-invariant measure $\lambda$. Then for the measure $\nu=\lambda*\mu\in P_\I(X)$ two cases are possible:

1) Every cell $A_i$ of the partition has $\bar\nu(A_i)=\sup_{x\in G}\nu(xA_i)\le\frac1n$. In this case we can proceed as in the proof of Theorem~\ref{t1} and prove that $\cov(\Delta_\I(A_i))\le n$ for all cells $A_i$ of the partition.

2) Some cell $A_i$ of the partition has $\bar\nu(A_i)>\frac1n$. In this case Lemma~\ref{l4} guarantees
that $\covJ(\Delta_\I(A_i))<n$ for the $G$-invariant Boolean ideal $\mathcal J\subset\PP(G)$ generated by the set $\{x\in G:\nu(xA_i)\le\frac1n\}$, and the set $M=\{x\in G:\mu(xA_i)>\frac1n\}$ does not belong to the ideal $\mathcal J$.
\end{proof}

Next, we extend Corollary~\ref{c3} to $G$-spaces with arbitrary (not necessarily countable) acting group $G$.
Given a $G$-space $X$ denote by $\HH$ the family of all countable subgroups of the acting group $G$.
 A subfamily $\F\subset \HH$ will be called
\begin{itemize}
\item {\em closed} if for each increasing sequence of countable subgroups $\{H_n\}_{n\in\w}\subset\F$ the union $\bigcup_{n\in\w}H_n$ belongs to $\F$;
\item {\em dominating} if each countable subgroup $H\in\HH$ is contained in some subgroup $H'\in\F$;
\item {\em stationary} if $\F\cap\C\ne\emptyset$ for every closed dominating subset $\C\subset\HH$.
\end{itemize}
It is known (see \cite[4.3]{Jech}) that the intersection $\bigcap_{n\in\w}\C_n$ of any countable family of closed dominating sets $\C_n\subset \HH$, $n\in\w$, is closed and dominating in $\HH$.

For a measure $\mu\in P(X)$ and a subgroup $H\in\HH$ let $$P_\I(H;\mu)=\{\lambda\in P(H):\forall x\in H\;\;\lambda*\delta_x*\mu\in P_\I(X)\}.$$

\begin{theorem}\label{t5} Let $(X,\I)$ be an ideal $G$-space and $\mu\in P(X)$ be a measure on $X$ such that the set $\HH_\I=\{H\in\HH:P_\I(H;\mu)\ne\emptyset\}$ is stationary in $\HH$. For any partition $X=A_1\cup\dots\cup A_n$ of $X$ either:
\begin{enumerate}
\item $\cov(\Delta_\I(A_i))\le n$ for all cells $A_i$ or else
\item $\covJ(\Delta_\I(A_i))<n$ for some cell $A_i$ and some $G$-invariant Boolean ideal $\mathcal J\subset\PP(G)$ such that $\{x\in G:\mu(xA_i)>\frac1n\}\notin\mathcal J$.
\end{enumerate}
\end{theorem}

\begin{proof} Let $\HH_\forall=\{H\in\HH_\I:\forall i\le n\;\;\cov(H\cap\Delta_\I(A_i))\le n\}$
and $\HH_\exists=\HH_\I\setminus\HH_{\forall}$. It follows that for every $H\in\HH_\forall$ and $i\in\{1,\dots,n\}$ we can find a subset $f_i(H)\subset H$ of cardinality $|f_i(H)|\le n$ such that $H\subset f_i(H)\cdot\Delta_\I(A_i)$. The assignment $f_i:H\mapsto f_i(H)$ determines a function $f_i:\HH_\forall\to [G]^{<\w}$ to the family of all finite subsets of $G$. The function $f_i$ is regressive in the sense that $f_i(H)\subset H$ for every subgroup $H\in \HH_\forall$.

By Corollary~\ref{c3}, for every subgroup $H\in\HH_\exists$, there are an index $i_H\in\{1,\dots,n\}$ and a finite subset $f(H)\subset H$ of cardinality $|f(H)|<n$ such that the set $J_H=H\setminus \big(f(H)\cdot(H\cap\Delta_\I(A_{i_H}))\big)$ generates the $H$-invariant ideal $\mathcal J_H\subset\PP(H)$ which does not contain the set $M_H=\{x\in H:\mu(xA_{i_H})>\frac1n\}$.
\smallskip

Since $\HH_\I=\HH_\forall\cup\HH_{\exists}$ is stationary in $\HH$, one of the sets $\HH_{\forall}$ or $\HH_\exists$ is stationary in $\HH$.
\smallskip

If the set $\HH_\forall$ is stationary in $\HH$, then by Jech's generalization \cite{Jech72}, \cite[4.4]{Jech}  of Fodor's Lemma, the stationary set $\HH_{\forall}$ contains another stationary subset $\mathcal S\subset\HH_\forall$ such that for every $i\in\{1,\dots,n\}$ the restriction $f_i|\mathcal S$ is a constant function and hence $f_i(\mathcal S)=\{F_i\}$ for some finite set $F_i\subset G$ of cardinality $|F_i|\le n$. We claim that $G=F_i\cdot\Delta_\I(A_i)$. Indeed, given any element $g\in G$, by the stationarity of $\mathcal S$ there is a subgroup $H\subset\mathcal S$ such that $g\in H$. Then $g\in H\subset f_i(H)\cdot\Delta_\I(A_i)=F_i\cdot\Delta_\I(A_i)$ and hence $\cov(\Delta_\I(A_i))\le|F_i|\le n$ for all $i$.
\smallskip

Now assume that the family $\HH_\exists$ is stationary in $\HH$. In this case for some $i\in\{1,\dots,n\}$ the set $\HH_i=\{H\in\HH_\exists:i_H=i\}$ is stationary in $\HH_\exists$.
Since the function $f:\HH_\exists\to [G]^{<\w}$ is regressive, by Jech's generalization \cite{Jech72}, \cite[4.4]{Jech}  of Fodor's Lemma, the stationary set $\HH_i$ contains another stationary subset $\mathcal S\subset\HH_i$ such that the restriction $f|\mathcal S$ is a constant function and hence $f(\mathcal S)=\{F\}$ for some finite set $F\subset G$ of cardinality $|F|<n$. We claim that the set $J=G\setminus (F\cdot\Delta_\I(A_i))$ generates a $G$-invariant ideal $\mathcal J$, which does not contain the set $M=\{x\in G:\mu(xA_i)>\frac1n\}$. Assume conversely that $M\in\mathcal J$ and hence $M\subset EJ$ for some finite subset $E\subset G$. By the stationarity of the set $\mathcal S$, there is a subgroup $H\in\mathcal S$ such that $E\subset H$. It follows $H\cap J=H\setminus \big(F\cdot (H\cap \Delta_\I(A_i))\big)=H\setminus \big(f(H)\cdot(H\cap\Delta_\I(A_{i_H}))\big)=J_H$ and
$$M_H=\{x\in H:\mu(xA_i)>\tfrac1n\}=H\cap M\subset H\cap EJ=E\mathcal J_H\in\mathcal J_H,$$
which contradicts the choice of the ideal $\mathcal J_H$.
\end{proof}

\section{Proof of Theorem~\ref{t2}}

Theorem~\ref{t2} is a simple corollary of Theorem~\ref{t5}.
Indeed, assume that $G=A_1\cup\dots\cup A_n$ is a partition of a group and $\I\subset\PP(G)$ is an invariant ideal on $G$ which does not contain some countable subset and hence does not contain some countable subgroup $H_0\subset G$. Let $\HH$ be the family of all countable subgroups of $G$ and  $\mu=\delta_1$ be the Dirac measure supported by the unit $1_G$ of the group $G$. We claim that that for every subgroup $H\in\HH$ that contains $H_0$ the set $P_\I(H;\mu)$ is not empty. It follows from $H_0\notin\I$ that the family $\I_H=\{H\cap A:A\in\I\}$
is an invariant Boolean ideal on the group $H$. Then the family $\{H\setminus A:A\in\I\}$ is a filter on $H$, which can be enlarged to an ultrafilter $\U_H$. The ultrafilter $\U_H$ determines a 2-valued measure $\mu_H:\PP(H)\to\{0,1\}$ such that $\mu_H^{-1}(1)=\U_H$. By the right invariance of the ideal $\I$, for every $A\in\I$ and $x\in H$ we get $\mu_H*\delta_x*\mu(A)=\mu_H(Ax)=0$, which means that $\mu_H\in P_\I(H;\mu)$. So, the set $\HH_\I=\{H\in\HH:P_\I(H;\mu)\ne\emptyset\}\supset\{H\in\HH:H\supset H_0\}$ is stationary in $\HH$.

Then by Theorem~\ref{t5} either
\begin{enumerate}
\item $\cov(\Delta_\I(A_i))\le n$ for all cells $A_i$ or else
\item $\covJ(\Delta_\I(A_i))<n$ for some cell $A_i$ and some $G$-invariant Boolean ideal $\mathcal J\subset\PP(G)$ such that $A_i^{-1}=\{x\in G:\delta_1(xA_i)>\frac1n\}\notin\mathcal J$.
\end{enumerate}

\end{document}